\documentclass{amsart}
\usepackage{mathabx}
\usepackage{amsmath,amscd,amsthm}
\usepackage{amsfonts,amssymb}
\usepackage[all]{xy}
\usepackage{mathtools}
\usepackage{enumerate}

\usepackage{hyperref}

\newcommand{\into}{\hookrightarrow}

\newcommand{\xto}{\xrightarrow}

\newcommand{\F}{\mathbb{F}}

\newcommand{\Ah}{\mathcal{A}}
\newcommand{\Bh}{\mathcal{B}}
\newcommand{\Ch}{\mathcal{C}}

\newcommand{\Sb}{\mathcal{S}}

\newcommand{\Ext}{\mathrm{Ext}}

\newcommand{\Hom}{\mathrm{Hom}}

\newcommand{\gExt}{\underline{\Ext}}

\newcommand{\gHom}{\underline{\Hom}}

\newcommand{\op}{\mathrm{op}}

\newcommand{\HH}{\mathrm{HH}}

\newcommand{\Imm}{\mathrm{Im}\,}

\newcommand{\Ker}{\mathrm{Ker}\,}

\newcommand{\bbb}{\bullet}

\newcommand{\Cb}{\Ch^\bbb}

\newcommand{\Hb}{H^\bbb}

\newcommand{\dee}{\partial}

\newcommand{\ot}{\otimes}

\newcommand{\Span}{\mathrm{Span}}

\theoremstyle{plain}

\newtheorem{theorem}{Theorem}[section]

\newtheorem{proposition}[theorem]{Proposition}
\newtheorem{lemma}[theorem]{Lemma}
\newtheorem{cor}[theorem]{Corollary}

\theoremstyle{definition}
\newtheorem{definition}[theorem]{Definition}

\newtheorem{notn}[theorem]{Notation}
\theoremstyle{remark}

\newtheorem{remark}[theorem]{Remark}

\input cyracc.def

\author{Eivind Xu Djurhuus}
\address{Department of Mathematics, UiO, Oslo, Norway.}
\email{eivindxd@uio.no}

\author{Gereon Quick}
\address{Department of Mathematical Sciences, NTNU, Trondheim, Norway}
\email{gereon.quick@ntnu.no}

\title[Non-realizability of a triple Massey product]{Non-realizability of a triple Massey product for the algebra \(\F_2[a,b,c]/(ab,bc)\)}

\begin{document}

\begin{abstract} 
We show that an often used example of a cohomology algebra with non-vanishing triple Massey product is intrinsically $A_3$-formal and therefore, in fact, cannot be realized as the cohomology of a differential graded algebra with non-vanishing triple Massey product. 
We prove this result by computing the graded Hochschild cohomology group which contains the potential obstruction to the vanishing. 
\end{abstract}

\maketitle

\section{Introduction}

Let $\Cb$ be a differential graded $\F_2$-algebra (DGA) with differential $\delta$ and cohomology algebra $\Hb$. 
Let $a,b,c$ be cohomology classes such that $a \cup b = 0$ and $b \cup c = 0$. 
We recall that the Massey product $\langle a,b,c \rangle$ is defined as follows.  
Let $A$, $B$, $C$ be cocycles representing $a$, $b$, $c$, respectively, 
and let $E_{ab}$ and $E_{bc}$ be cochains such that $\delta E_{ab} = A \cup B$ and 
$\delta E_{bc} = B \cup C$.  
The set $M \coloneqq \{A,B,C,E_{ab},E_{bc}\}$ is called a defining system for the triple Massey product of $a$, $b$, and $c$. 
The cochain $A \cup E_{bc} + E_{ab} \cup C$ is a cocycle. 
We write $\langle a,b,c \rangle_M \in \Hb$ for the corresponding  cohomology class. 
The {\em triple Massey product} $\langle a,b,c \rangle$ is the set of all cohomology classes $\langle a,b,c \rangle_M$ for all such defining systems $M$. 
The class $\langle a,b,c \rangle_M$ depends on the choice of the defining system $M$. 
The image in the quotient $\Hb/(a \cup \Hb + \Hb \cup c)$, however, 
is uniquely determined by $a$, $b$, and $c$. 
We say that $\langle a,b,c \rangle$ {\em vanishes} if its corresponding class in $\Hb/(a \cup \Hb + \Hb \cup c)$ is zero. 
Massey products play an important role in the classification of DGAs with a given cohomology ring.

To construct the simplest commutative graded algebra which may be realized as the cohomology of a DGA 
with a non-vanishing Massey product, 
one may consider $\F_2[a, b, c]/(ab, bc)$ with $a,b,c$ elements in degree one. 
We have $ab=0$ and $bc=0$ by construction, 
and hence $\langle a,b,c \rangle$ is defined. 
We may then expect that it may be possible 
to find a DGA such that $\langle a,b,c \rangle$ does not vanish.  
In this note, however, we show that the Massey product $\langle a,b,c \rangle$ always vanishes. 
More precisely, our main result is the following theorem. 
Recall that a differential graded algebra is called {\em $A_3$-formal} if the minimal $A_{\infty}$-model of $\Cb$  has a trivial homotopy associator $m_3$ (see e.g.~\cite{Keller}, \cite{PQ2}, and Section \ref{sec:HH}). 

\begin{theorem}\label{thm:main_intro}
Let $\Cb$ be a differential graded algebra over $\F_2$ with cohomology algebra 
isomorphic to $\F_2[a, b, c]/(ab, bc)$. 
Then $\Cb$ is $A_3$-formal. 
Thus, all triple Massey products for $\Cb$ vanish. 
In particular, the triple Massey product 
$\langle a, b, c \rangle$ vanishes for $\Cb$.  
\end{theorem}

Our interest in the realizability of triple Massey products grew out of the work of Hopkins--Wickelgren in \cite{HW} on triple Massey products in Galois cohomology. 
The latter has inspired a lot of research in recent years, see for example \cite{MS2} and \cite{MT2}. 

\begin{remark}\label{rem:smaller_algebras}
One may consider other $\F_2$-algebras and ask whether they realize a non-trivial Massey product.  
Since the definition of the Massey product does not require distinct elements, 
we may first consider the algebra $\F_2[a]/(a^2)$ with just one generator. 
However, $\F_2[a]/(a^2)$ is zero in degree two, 
and hence $\langle a,a,a \rangle$ must vanish.   
The algebra $\F_2[a,b]/(ab)$ is a Boolean graded algebra in the sense of \cite[Definition 6.1]{PQ1}. 
More generally, any algebra of the form $\F_2[a_1,\ldots,a_n]/I$ 
where $I$ is the ideal generated by all products $a_ia_j$ for $i \ne j$ 
is a Boolean graded algebra. 
The algebra $\F_2[a,b]/(a^2,ab)$ is a connected sum of a dual and a Boolean graded algebra in the sense of \cite[Section 6]{PQ1}. 
All these algebras are intrinsically $A_{\infty}$-formal by \cite[Theorem 7.13]{PQ1} 
and do not allow for non-vanishing Massey products. 
\end{remark}

We now outline the proof of Theorem \ref{thm:main_intro} and thereby describe the content of the paper. 
For the whole manuscript, 
we assume that all algebras and vector spaces are over $\F_2$. 
The differentials in complexes of algebras and modules, except the bar complex, raise the degree by one. 
In Section \ref{sec:HH} we recall the Hochschild cohomology of graded algebras and construct the Hochschild cohomology class $[m_3] \in \HH^{3,-1}(\Hb(\Cb))$ associated to a differential graded algebra $\Cb$. 
We note that $[m_3]$ equals the canonical class of $\Cb$ introduced by Benson--Krause--Schwede in \cite{BKS} as an obstruction for the realizability of modules over Tate cohomology.  
We then show that $\Cb$ is $A_3$-formal if and only if $[m_3]$ is zero. 
For the latter, we assume familiarity with some basic theory of $A_{\infty}$-algebras.  
In Section \ref{sec:Koszul_algebras} we recall the definition of Koszul algebras and show that $\F_2[a, b, c]/(ab, bc)$ is Koszul. 
Knowing that an algebra is Koszul simplifies the task to compute its Hochschild cohomology significantly. 
In Section \ref{sec:main_result_proof} we prove Theorem \ref{thm:computationof_HH} which states that $\HH^{3,-1}(\F_2[a, b, c]/(ab, bc)) = 0$ by computing the image and the kernel of the differential in the Hochschild complex. 
Theorem \ref{thm:computationof_HH} then implies Theorem \ref{thm:main_intro}. 


\subsection*{Acknowledgement} 
We thank Mads Hustad Sandøy for helpful conversations. 
We are grateful to the anonymous referee for comments and suggestions that helped to improve the manuscript. 


\section{Hochschild cohomology and Massey products}\label{sec:HH}

Let $A$ be a graded unital $\F_2$-algebra. 
We recall that the bar resolution
$B(A)$ of $A$ is the non-negative chain complex
of free graded $A$-bimodules
given by $B_n(A) \coloneqq A^{\otimes n+2}$ for $n \geq 0$.
The differential 
$d_n \colon B_n(A) \to B_{n-1}(A)$
is given by
\begin{align}\label{eq:diff_bar_complex}
a_0 \otimes \cdots \otimes a_{n+1} 
\mapsto \sum_{i = 0}^n 
a_0 \otimes \cdots \otimes 
a_i a_{i+1} \otimes \cdots \otimes a_{n+1}.
\end{align}
We write $A^e = A \ot A^{\op}$. 
Note that $A^{\otimes n+2} \cong A^e \otimes A^{\otimes n}$
as a graded $A$-bimodule, hence $B(A)$
indeed consists of free modules. 

\begin{proposition}\label{bar-acyclic}
The bar resolution $B(A)$ is a free resolution
of $A$ as a graded $A$-bimodule.
\end{proposition}
\begin{proof}
It suffices to show that 
the extended complex $\widetilde{B}(A)$ is acyclic,  
where $\widetilde{B}(A)$  
is extended from $B(A)$ by adjoining 
$\widetilde{B}_{-1}(A) := A$ 
in degree $-1$ via the multiplication 
map $\mu \colon A \otimes A \to A$. 
We claim that the map
$h \colon \widetilde{B}(A) \to \widetilde{B}(A)$ of degree $1$ given by
\[
a_0 \otimes \cdots \otimes a_{n+1}
\mapsto 1 \otimes a_0 \otimes \cdots \otimes a_{n+1}
\]
is a contracting homotopy i.e., 
$dh + hd = 1$. 
Indeed, we compute directly that
\[
dh(a_0 \otimes \cdots \otimes a_{n+1})
= a_0 \otimes \cdots \otimes a_{n+1}
- hd(a_0 \otimes \cdots \otimes a_n). \qedhere
\]
\end{proof}

\begin{definition}
Let $M$ and $N$ be graded $A$-bimodules. 
We define $\gHom_A(M, N)$ as the graded $\F_2$-vector space
with degree $s$ component given by
$A$-linear graded maps $f \colon M \to N[s]$, where $N[s]$ is the graded $A$-module given by $N[s]^n = N^{s+n}$.
\end{definition}

\begin{definition}
Let $M$ be a graded $A$-bimodule. 
We define the Hochschild cohomology $\HH^{n,\bullet}(A,M)$
as the $n$th cohomology of the cochain complex
\[
\gHom_{A^e}(B(A),M)
\]
of graded $\F_2$-vector spaces with differential $\delta^n (f) = f \circ d_n$ where $d_n$ is the differential of $B(A)$. 
When $M=A$ we will write
$\HH(A) \coloneqq \HH(A, A)$. 
\end{definition}

We note that the groups $\HH^{\bullet, \bullet}(A,M)$ are equipped with a cohomological grading, and an internal grading induced by the grading of $A$ and $M$.
We can describe $\HH^{n,s}(A,M)$ more concretely as follows. 
Using the natural contracting isomorphism
\[
\gHom_{A^e}(A^e \otimes A^{\otimes n}, M) \cong 
    \gHom_{\F_2}(A^{\otimes n}, M)
\]
we see that $\HH^{n,\bullet}(A, M)$ is isomorphic to the $n$th cohomology of the complex
\begin{align}
\label{eq:HH_reduced_complex}
\cdots \to \gHom_{\F_2}(A^{\otimes n-1}, M) \xto{\partial}
\gHom_{\F_2}(A^{\otimes n}, M) 
\xto{\partial} \gHom_{\F_2}(A^{\otimes n+1}, M) \to \cdots,
\end{align}
where the differentials are given by
\begin{align*}
\partial(f)(a_1 \otimes \cdots \otimes a_{n+1})
    & = 
    a_1f(a_2 \otimes \cdots \otimes a_{n+1})\\
    & + \sum_{i=1}^{n} 
        f(a_1 \otimes \cdots \otimes a_i a_{i+1}
        \otimes \cdots \otimes a_{n+1})\\
    & + 
    f(a_1 \otimes \cdots \otimes a_{n})a_{n+1}.
\end{align*}


\begin{remark}\label{rem:HH_and_Ext}
By Proposition~\ref{bar-acyclic},
we see that
$\HH(A, M)$ computes the graded Ext
modules $\gExt_{A^e}(A, M)$. 
In particular,
we can compute $\HH(A, M)$
using any free resolution of $A$ as
a graded $A$-bimodule. 
\end{remark}


We now assume that the reader is familiar with $A_{\infty}$-algebras. 
For an introduction to the theory of $A_{\infty}$-algebras and references with all details we refer to \cite{Keller}. 
Let $(\Cb, \delta, \cup)$ be a differential graded algebra (DGA) over $\F_2$ with cohomology algebra $\Hb$. 
%
By the work of Kadeishvili \cite{Kadeishvili82, Kadeishvili88} (see also \cite{Kadeishvili23}, \cite{Keller}, and \cite{Merkulov}), 
one can equip $\Hb$ with the structure of an $A_{\infty}$-algebra $(\Hb,\{m_n\}_{n\ge 1})$ such that $m_1=0$ together with a quasi-isomorphism of $A_{\infty}$-algebras 
$(\Hb,\{m_n\}) \xto{\simeq} (\Cb, \delta, \cup)$. 
The $A_{\infty}$-algebra $(\Hb,\{m_n\})$ is called a {\em minimal model} of $(\Cb, \delta, \cup)$.  
Since any two minimal models of $(\Cb, \delta, \cup)$ are isomorphic as $A_{\infty}$-algebras, 
we speak of {\em the} minimal model from now on. 
A DGA is called {\em $A_{\infty}$-formal}, or {\em formal} as an $A_{\infty}$-algebra, if its minimal model can be chosen such that $m_n = 0$ for all $n \ge 3$. 
%
%
We now consider the following weaker notion. 

\begin{definition}\label{def:A3formal}
Let $\Cb$ be a DGA. 
We say that $\Cb$ is {\em $A_3$-formal} if its minimal model can be chosen such that $m_3 = 0$.  
\end{definition}

We refer to \cite{PQ2} for a more detailed discussion of $A_3$-formality. 
We recall the following special case from \cite[Theorem C]{BMM}. 
Let $a,b,c \in \Hb$ be cohomology classes such that $a \cup b = 0$ and $b \cup c = 0$. 
Then $m_3(a \ot b \ot c) \in \langle a,b,c \rangle$. 
This implies the following well-known fact. 

\begin{proposition}\label{prop:A3formal_Massey}
Let $\Cb$ be a DGA with cohomology algebra $\Hb$.   
Let $a,b,c \in \Hb$ be cohomology classes such that $a \cup b = 0$ and $b \cup c = 0$. 
Assume that $\Cb$ is $A_3$-formal. 
Then the triple Massey product $\langle a,b,c \rangle$ contains zero. \qed
\end{proposition}

Let $\Cb$ be a DGA with cohomology algebra $\Hb$. 
We note that $m_3 \colon (\Hb)^{\ot 3} \to \Hb[-1]$ is a graded map 
which we can construct as follows (see for example \cite[Section 5]{BKS}, \cite{Merkulov}, \cite{PQ2}).  
We choose an $\F_2$-linear graded map $f_1 \colon \Hb \to \Ker \delta$ which induces the identity when taking cohomology.  
Since $f_1$ is multiplicative on cohomology, we can find a graded $\F_2$-linear map  
$f_2 \colon \Hb \otimes \Hb \to \Cb$ of degree $-1$ satisfying  
\begin{align*}
\delta(f_2(a \ot b)) = f_1(a \cup b) + f_1(a) \cup f_1(b). 
\end{align*}
%
Now we define a graded $\F_2$-linear map $\Phi_3 \colon (\Hb)^{\otimes 3} \to \Cb[-1]$ by 
\begin{align}\label{eq:def_of_Phi3_minimal}
\Phi_3(a \ot b \ot c) \coloneqq 
f_1(a) f_2(b \ot c) 
+ f_2(a \ot b) f_1(c) + f_2((a b) \ot c + a \ot (b c)) 
\end{align}
for all homogeneous elements $a,b,c \in \Hb$ where we write $xy$ for the product $x \cup y$ to shorten the notation.  
We check that $\Phi_3$ has image in the cocycles of $\Cb$, 
and hence $\Phi_3$ induces a graded map $[\Phi_3] \colon (\Hb)^{\ot 3} \to \Hb[-1]$. 
We set $m_3 := [\Phi_3]$. 
By \cite[Proposition 5.4]{BKS}, $m_3$ is a cocycle in the complex \eqref{eq:HH_reduced_complex}.   
By \cite[Corollary 5.7]{BKS}, 
the corresponding Hochschild cohomology class  $[m_3] \in \HH^{3,-1}(\Hb)$ is independent of the choice of $f_1$ and $f_2$,   
and it is called the {\em canonical class} of $\Cb$ following Benson--Krause--Schwede who studied this class as an obstruction to the realizability of modules over Tate cohomology in \cite{BKS}.   
The following result is a modified version of Kadeishvili's theorem \cite{Kadeishvili88} (see also \cite[Theorem 3.9]{PQ2}, and 
\cite[Theorem 4.7]{ST}). 

\begin{theorem}\label{thm:canonical_A3formal}
Let $\Cb$ be a DGA 
with canonical class $[m_3] \in \HH^{3,-1}(\Hb)$. 
Then $\Cb$ is $A_3$-formal if and only if $[m_3]=0$. 
\end{theorem}
\begin{proof}
If $\Cb$ is $A_3$-formal then $m_3$ is a trivial cocycle, 
and the class of $m_3$ vanishes in $\HH^{3,-1}(\Hb)$. 
%
Now we assume that $[m_3]=0$  in $\HH^{3,-1}(\Hb)$. 
We may assume that $\Phi_3$ and hence $m_3$ is constructed using maps $f_1,f_2$ as in \eqref{eq:def_of_Phi3_minimal}. 
Then there exists an $\F_2$-linear map $\eta \colon (\Hb)^{\otimes 2} \to (\Ker \delta)[-1]$ 
such that $\dee^2 [\eta] = m_3$ as maps 
$(\Hb)^{\otimes 3} \to \Hb[-1]$. 
We set $\tilde{f}_2 = f_2 + \eta$. 
We note that $\tilde{f}_2$ satisfies 
\begin{align*}
\delta\tilde{f}_2(a \ot b) = \delta(f_2(a \ot b) + \eta(a \ot b))  
= f_1(a \cup b) + f_1(a) \cup f_1(b)
\end{align*}
since $\delta \circ \eta=0$. 
We then define the map $\widetilde{\Phi}_3$ by replacing $f_2$ with $\tilde{f}_2$, 
i.e., we define 
\begin{align*}
\widetilde{\Phi}_3(a \ot b \ot c) \coloneqq 
f_1(a) \tilde{f}_2(b \ot c) 
+ \tilde{f}_2(a \ot b) f_1(c) + \tilde{f}_2((a b) \ot c + a \ot (b c)) 
\end{align*}
for all homogeneous elements $a,b,c \in \Hb$ where we again write $xy$ for $x \cup y$ to shorten the notation.  
We then have 
\begin{align*}
(\Phi_3 - \widetilde{\Phi}_3)(a \ot b \ot c)
=  f_1(a) \eta(b \ot c) + \eta(a \ot b) f_1(c) + \eta((a b) \ot c + a \ot (b c)). 
\end{align*}
By definition of $\dee^2$ and the assumption on $\eta$, 
this implies $\widetilde{\Phi}_3 = \Phi_3 - \dee^2\eta = 0$ as maps $(\Hb)^{\otimes 3} \to \Hb[-1]$. 
\end{proof}

As a direct consequence we get: 

\begin{cor}\label{cor:HH_and_A3_formal}
Let $A$ be a graded algebra with $\HH^{3,-1}(A)=0$. 
Then every DGA $\Cb$ whose cohomology algebra is isomorphic to $A$ is $A_3$-formal. \qed
\end{cor}

By Proposition \ref{prop:A3formal_Massey} and Corollary \ref{cor:HH_and_A3_formal},   
in order to show Theorem \ref{thm:main_intro} it will suffice to show $\HH^{3,-1}(\F_2[a, b, c]/(ab, bc))=0$. 
This is what we now set out to prove.


\section{Koszul algebras}\label{sec:Koszul_algebras}

Let $V$ denote a finite-dimensional $\F_2$-vector space, 
and let $T(V)$ denote its graded tensor algebra over $\F_2$.  
For $R \subseteq V \otimes V$, 
let $(R)$ denote the two-sided ideal in $T(V)$ 
generated by $R$. 
We recall from \cite{ppqa} that a graded algebra 
of the form $T(V)/(R)$ is called a {\em quadratic algebra}. 
%
%
For any quadratic algebra we can define the following
chain complex of
free graded $A$-bimodules.

\begin{definition}\label{koszulcomplex}
Let $A = T(V)/(R)$ be a quadratic algebra. 
For $n\geq 0$ and $1 \leq i \leq n-1$, let
\begin{align}\label{eq:X_is}
X_i^n \coloneqq 
V^{\otimes i-1} \otimes R \otimes V^{\otimes n-i-1} \subseteq V^{\otimes n}     
\end{align}
and 
\[
K'_n \coloneqq \bigcap_{i=1}^{n-1} X_i^n\subseteq V^{\otimes n}.
\]
Here we interpret the empty intersection as the whole space,
i.e., $K'_0 = \F_2$ and $K'_1 = V$.
The Koszul complex $K(A^e, A)$ of $A$ is defined
as the non-negative chain complex of graded $A$-bimodules with
\[
K_n(A^e, A) = A \otimes K'_n \otimes A,
\]
and differential $d_n$ induced by the one in the bar resolution $B(A)$, 
i.e., 
\[
d_n \colon a \otimes v_1 \otimes \cdots \otimes v_n \otimes b
\mapsto
av_1 \otimes v_2 \otimes \cdots \otimes v_n \otimes b
+ a \otimes v_1 \otimes \cdots \otimes v_n b.
\]
\end{definition}

Note that for the differential $d_n$ in the Kozul complex, 
the middle terms of the bar construction differential, 
see Equation \eqref{eq:diff_bar_complex}, vanish. 
This is because each product \(v_i v_{i+1}\)
in a middle term has its factors \(v_i\), \(v_{i+1}\) in the space of relations $R$, 
so the product \(v_i v_{i+1}\) vanishes in $A$, 
which is where the product in the expression 
of the differential is taking place. 


\begin{definition}
A quadratic algebra $A$ is called {\em Koszul}
if its Koszul complex $K(A^e, A)$ is
a resolution of $A$ as a graded $A$-bimodule,
i.e., if $H_n(K(A^e, A)) = 0$ for $n > 0$
and $H_0(A^e, A) = A$.
\end{definition}


We will now show that $A = \F_2 [a, b, c]/(ab, bc)$ is Koszul. 
Consider the $\F_2$-vector spaces 
$V \coloneqq \Span_{\F_2}\{a, b,c\}$ 
and 
\[
R \coloneqq \Span_{\F_2} 
\{a\otimes b, b\otimes a, c \otimes b, b \otimes c, a\otimes c + c \otimes a\}
\subseteq V \otimes V,
\]
chosen such that we can identify $A$ with
$T(V)/(R)$. 
In particular, we see that the algebra 
$A$ is quadratic. 


Let $e \coloneqq a \otimes c + c \otimes a$. 
To find a convenient linear basis for $V^{\otimes n}$, 
we introduce the set 
$\mathcal{B}^n$ consisting of strings 
$x = x_1 \cdots x_k$
of symbols from the set $\{ a, b, c, e\}$ such that
$|x_1| + \cdots + |x_k| = n$ and
$ca$ does not occur as a substring of $x$. Here
$|x_i|$ denotes the degree of the symbol $x_i$,
so $|a|=|b|=|c|=1$ and $|e|=2$. We identify
the strings in $\mathcal{B}^n$ with tensors in $V^{\otimes n}$.
As an example we see that
\[\mathcal B^2 = \{a \otimes a, a \otimes b, a \otimes c,
    b \otimes a, b \otimes b, b \otimes c, c \otimes b,
    c \otimes c, e\}.\]

\begin{lemma}\label{lemma:basis_of_Vn}
For each $n$, the set $\Bh^n$ is a basis
for $V^{\otimes n}$. 
\end{lemma}
\begin{proof}
From the basis
$\{a, b, c\}$ of $V$ we obtain a standard
basis of $V^{\otimes n}$ consisting
of the pure tensors in the symbols
$\{a, b, c\}$. 
We will show that we can obtain $\mathcal B^n$
from this standard basis using only elementary
column operations, from which it follows that $\mathcal B^n$
is also a basis. 
Starting from the standard basis,
we can replace each occurrence of $c \otimes a$
with $e$ by adding a suitable linear combination
of the standard pure tensors. 
More formally,
we do this using induction.

For $i \geq 0$, let
$\mathcal B^{n, i}$ be the set of strings $x_1 \cdots x_k$
in the symbols $\{a, b, c, e\}$ satisfying the conditions:
\begin{itemize}
\item $|x_1| + \cdots + |x_k| = n$,
\item there are at most $i$ occurrences of $e$,
\item the substring $ca$ only occurs after all the $e$'s,
\item if there are less than $i$ occurrences of $e$, 
there are no occurrences of $ca$.
\end{itemize}
We see that $\mathcal{B}^{n,0}$ is the standard
basis, while $\mathcal{B}^{n, i} = \mathcal B^n$
for large enough $i$ (e.g., for $i \geq n/2$). 
We will show that each $\mathcal{B}^{n,i+1}$
can be obtained from $\mathcal{B}^{n,i}$
using only elementary column operations, where
we have identified the strings with tensors in $V^{\ot n}$.   
We obtain $\mathcal B^{n, i+1}$ from
$\mathcal B^{n,i}$ in the following manner.
If $x \in \mathcal B^{n,i}$ has no
occurrences of $ca$, then
we already have $x \in \mathcal{B}^{n,i+1}$,
so we do nothing. 
Otherwise, there are precisely $i$ occurrences of $e$ and
at least one occurrence of $ca$ in $x$.
Consider the string $x'$ obtained by
replacing the first occurrence of $ca$
with $ac$. 
We see that $x'\in \mathcal B^{n, i}$
and we add $x'$ to $x$
to obtain the tensor $x' + x \in \mathcal{B}^{n, i+1}$.
All tensors in $\mathcal{B}^{n, i+1}$ can be obtained
in precisely one of these two ways, and thus
we obtain $\mathcal{B}^{n, i+1}$ from $\mathcal{B}^{n,i}$
as wanted.
\end{proof}

We now define certain subsets of $\mathcal{B}^n$ which we will show are bases for the spaces $X_i^n$ 
from Equation \eqref{eq:X_is}.
For
$1 \leq i \leq n-1$, let $\mathcal B_i^n$  consist of the strings in $\mathcal B^n$
where
the $(i, i+1)$-part is in $R$, where we count
$e$ with multiplicity $2$. More precisely,
for a string $x \in \mathcal B^n$ and 
an integer $i$, $1 \leq i \leq n$, we 
obtain a symbol $x_{(i)}$ by the following process. We first
modify the string $x$ into a string $x'$ of length $n$
by doubling every occurrence of $e$ in $x$, and
we then set $x_{(i)} = x'_i$. We can now define 
\[ \mathcal B_i^n \coloneqq
    \{x \in \mathcal{B}^n \mid
    x_{(i)}x_{(i+1)} \in 
    \{ab, ba, cb, bc,
     eb, be, ee\}\}.
\]
For instance, for the string $x = bebe$ we get that
$x' = beebee$ such that, e.g., 
$x_{(4)} = b$, and we see that
$x \in \mathcal{B}_i^6$ for all $1 \leq i \leq 5$. 

Now we can prove the main result of this section. 

\begin{proposition}\label{prop:A_is_Koszul}
The quadratic algebra $A = \F_2 [a, b, c]/(ab, bc)$
is Koszul.    
\end{proposition}

\begin{proof}
We will show that, for each $n \geq 0$, 
the basis $\mathcal B^n$ of $V^{\otimes n}$ distributes over 
the subspaces $X_1^n, \ldots, X_{n-1}^n$, 
i.e., 
for each $X^n_i$ the 
subset $\Bh^n_i \subseteq \Bh^n$ 
forms a basis for $X^n_i$.  
By the work of Backelin in \cite{Backelin}, 
see also \cite[Chapter 2, Theorem 4.1]{ppqa}, 
this implies that $A$ is Koszul.  
By definition, 
${X_i^n = V^{\otimes i-1} \otimes R \otimes V^{n-i-1}}$,
hence from the standard basis of $V$ and
the basis $\mathcal R = \{a \otimes b, b\otimes a, 
c \otimes b, b \otimes c, e\}$
of $R$, we obtain a basis of $X_i^n$. This basis
consists of strings $x$ in the symbols
$\{a, b, c, e\}$ with $e$ only possibly occurring as the $i$-th symbol and with
$x_{(i)}x_{(i+1)} \in \mathcal R$,
where $x_{(i)}$ is the notation introduced above to
define $\mathcal B_i^n$.
Using a similar induction argument as in the proof of Lemma \ref{lemma:basis_of_Vn}, we
replace each occurrence of $ca$ with $e$ in
this basis using
only elementary column operations to obtain a new
basis for $X_i^n$. This
gives precisely the set 
$\mathcal B_i^n \subseteq \mathcal B^n$,
which shows that $\mathcal B^n$ distributes
$X_1^n, \ldots, X_{n-1}^n$. 
\end{proof}

\begin{remark}
As a consequence of the proof of Proposition \ref{prop:A_is_Koszul} 
we see that
$\Bh'_n \coloneqq \bigcap_{i=1}^{n-1} \mathcal B_i^n$
is a basis
for
$K'_n = \bigcap_{i=1}^{n-1} X_i^n$.  
This basis $\mathcal B'_n$ can be explicitly
described as the set of strings $x_1 \cdots x_k$
in the symbols
$\{a, b, c, e\}$ satisfying
that
$|x_1|+ \cdots + |x_k|= n$
and that the symbols in the string alternate
between $b$ and a symbol
from the set $\{a, c, e\}$.
For example, we get 
\begin{align}
\label{eq:B'3}
\Bh_3' = \{
    aba, abc,
    cba, cbc,
    bab, bcb,
    eb, be\}
\end{align}
and 
\begin{align}
\label{eq:B'4}
\Bh_4' = \{
    abab, abcb, bcba, bcbc, 
    baba, babc, cbab, cbcb, 
    abe, eba, cbe, ebc, beb\}.
\end{align}
\end{remark}


\section{Proof of the main result}\label{sec:main_result_proof}

By Proposition \ref{prop:A3formal_Massey} and Corollary \ref{cor:HH_and_A3_formal},   
Theorem \ref{thm:main_intro} will follow from the following: 

\begin{theorem}\label{thm:computationof_HH}
We have $\HH^{3,-1}(\F_2[a, b, c]/(ab, bc))=0$.
\end{theorem}
\begin{proof}
Since $A$ is Koszul, 
the natural inclusion $K(A^e, A) \into B(A)$ 
is a quasi-isomorphism. 
Hence we can compute
$\HH(A)$ as the cohomology of
the complex $\gHom_{A^e}(K(A^e, A), A)$.
We first observe that we have
\[
K(A^e, A)_n = A \otimes K'_n \otimes A \cong A^e \otimes K'_n,
\]
as graded $A$-bimodules. 
Using the contracting isomorphism
\[
\gHom_{A^e} (A^e \otimes K'_n, A) \cong \gHom_{\F_2}(K'_n, A)
\]
of graded vector spaces 
we see that $\HH(A)$ can be computed as the cohomology
of the following complex of graded vector spaces:
\[
\cdots \to \gHom_{\F_2}(K'_{n-1}, A) 
\xrightarrow{\partial^{n-1}} \gHom_{\F_2}(K'_n, A)
\xrightarrow{\partial^{n}} \gHom_{\F_2}(K'_{n+1}, A)
\to \cdots 
\]
where the differential $\partial^n$
is given by
\[
\partial^n(f)(v_1 \otimes \cdots \otimes v_{n+1}) 
    = v_1f(v_2 \otimes \cdots \otimes v_{n+1}) + 
    f(v_1 \otimes \cdots \otimes v_n)v_{n+1}.
\]
To compute $\HH^{3,-1}(A)=0$,
we need to show that  
\[
\Hom_{\F_2}(K'_2, A^1) 
    \xrightarrow{\partial^2} \Hom_{\F_2}(K'_3, A^2)
    \xrightarrow{\partial^3} \Hom_{\F_2}(K'_4, A^3)
\]
is exact in the middle, 
i.e., $\Imm \partial^2 = \Ker \partial^3$. 


First we will describe $\Ker \partial^3$. 
To do so, we use the following notation for elements 
in the bases $\mathcal B'_3$ and $\mathcal B'_4$ of $K'_3$ and $K'_4$, respectively. 
Since
$a$ and $c$ play symmetrical roles in $A$,
we will introduce the notation $(a|c)$ to mean
that each of $a$ and $c$ can be used in the expression.
For example, for a map $f \in \Hom_{\F_2}(K'_3, A^2)$, 
the equation $f((a|c) \otimes b) = 0$ would
mean that we have two equations $f(a \otimes b) = 0$ 
and $f(c \otimes b) = 0$. If there are several instances
of $(a|c)$ in the expression, each instance
can be replaced by $a$ or $c$ independently of each other.

\begin{lemma}\label{kernel-rels}
A map $f \in \Hom_{\F_2}(K'_3, A^2)$
lies in $\Ker \partial^3$ if and only if it
satisfies the following relations:
\begin{align}
f(b \otimes (a|c) \otimes b) \in \Span_{\F_2}\{ b^2 \}, \label{i} \tag{i}\\
f((a|c) \otimes b \otimes (a|c)) \in \Span_{\F_2}\{ a^2, ac, c^2\}, \label{ii} \tag{ii} \\
a(f(c \otimes b \otimes a) + f(e \otimes b)) 
    + cf(a \otimes b \otimes a) = 0, \label{iii} \tag{iii} \\
c(f(a \otimes b \otimes c) + f(e \otimes b))
    + af(c \otimes b \otimes c) = 0, \label{iv} \tag{iv} \\
a(f(a \otimes b \otimes c) + f(b \otimes e))
    + cf(a \otimes b \otimes a) = 0, \label{v} \tag{v}\\
c(f(c \otimes b \otimes a) + f(b \otimes e))
    + af(c \otimes b \otimes c) = 0, \label{vi} \tag{vi} \\
f(e \otimes b) + f(b\otimes e) \in \Span_{\F_2}\{ a^2, ac, c^2 \}. \label{vii} \tag{vii}
\end{align}
\end{lemma}
\begin{proof}
A map $f \in \Hom_{\F_2}(K'_3, A^2)$ is in the kernel of $\partial^3$ 
if and only if $\partial^3(f)$ vanishes on all 
elements of the basis $\Bh'_4$. 
We now evaluate $\partial^3(f)$ on $\Bh'_4$ as described in \eqref{eq:B'4}. 
First, 
since $(a|c)b = 0$ in $A^2$, 
the equation 
\[
\partial^3(f)((a|c)\ot b \ot (a|c) \ot b) 
= (a|c) f(b \ot (a|c) \ot b) + f((a|c) \ot b \ot (a|c)) b
= 0
\]
implies 
$f(b \otimes (a|c) \otimes b) \in \Span_{\F_2}\{b^2\}$, 
and
$f((a|c) \otimes b \otimes (a|c)) \in \Span_{\F_2}\{a^2, ac, c^2\}$.  
The equation $\partial^3(f)(b \otimes (a|c)\otimes b \otimes (a|c))=0$
gives the same relations. 
This shows that \eqref{i} and \eqref{ii} are necessary and sufficient. 
Second, 
\eqref{iii} and \eqref{iv} are imposed by 
the equations  
\[
\partial^3(f)(e \ot b \ot a) 
= af(c \ot b \ot a) + cf(a \ot b \ot a) + f(e \ot b)a 
= 0
\]
and  
\[
\partial^3(f)(e \ot b \ot c) 
= af(c \ot b \ot c) + cf(a \ot b \ot c) + f(e \ot b)c
= 0.
\]
Similarly, 
\eqref{v} and \eqref{vi} are imposed by 
the equations  
\[
\partial^3(f)(a \ot b \ot e)
= a f(b \ot e) + f(a \ot b \ot a)c + f(a \ot b \ot c)a 
= 0
\]
and 
\[
\partial^3(f)(c \ot b \ot e)
= c f(b \ot e) + f(c \ot b \ot a)c + f(c \ot b \ot c)a 
= 0. 
\]
Finally, 
the condition
\[
\partial^3(f)(b \ot e \ot b) 
= bf(e \ot b) + f(b \ot e)b
= 0 
\]
gives the relation 
$f(e \ot b) + f(b \ot e) \in \Span_{\F_2}\{ a^2, ac, c^2 \}$ 
which is \eqref{vii}. 
\end{proof}

\begin{notn}
For $v \in \mathcal B'_n$ and $x \in A^i$, 
we write
$F_n(v;x)$ for
the map in $\Hom_{\F_2}(K'_n, A^i)$ sending
$v$ to $x$ and other basis vectors in $\mathcal B'_n$ 
to zero. 
\end{notn}

\begin{lemma}\label{lemma:basis_of_Kerd3}
The set of maps 
\begin{align*}
\Sb_3 \coloneqq 
\Big\{
\partial^2(F_2(b \ot a;b)) & = F_3(b \ot a \ot b; b^2), \\
\partial^2(F_2(b \ot c;b)) & = F_3(b \ot c \ot b; b^2), \\
\partial^2(F_2(e; b)) & = F_3(b \ot e; b^2) + F_3(e \ot b; b^2), \\
\partial^2(F_2(c\ot b; a)) & = F_3(e \ot b; a^2) + F_3(c \ot b \ot a; a^2) +  F_3(c \ot b \ot c; ac), \\
\partial^2(F_2(b \ot c; a)) & = F_3(b \ot e; a^2) + F_3(a \ot b \ot c; a^2) +  F_3(c \ot b \ot c; ac), \\
\partial^2(F_2(a \ot b; a)) & = F_3(e \ot b; ac) + F_3(a \ot b \ot c; ac) +  F_3(a \ot b \ot a; a^2), \\
\partial^2(F_2(a \ot b; c)) & = F_3(e \ot b; c^2) + F_3(a \ot b \ot c; c^2) +  F_3(a \ot b \ot a; ac), \\
\partial^2(F_2(b \ot a; c)) & = F_3(b \ot e; c^2) + F_3(c \ot b \ot a; c^2) +  F_3(a \ot b \ot a; ac), \\
\partial^2(F_2(c \ot b; c)) & = F_3(e \ot b; ac) + F_3(c \ot b \ot a; ac) +  F_3(c \ot b \ot c; c^2), \\
\partial^2(F_2(b \ot c; c)) & = F_3(b \ot e; ac) + F_3(a \ot b \ot c; ac) +  F_3(c \ot b \ot c; c^2)
\Big\}
\end{align*}
is a basis of $\Ker \partial^3$. 
\end{lemma}
\begin{proof}
We consider the set $\Sb'_3$ of $\F_2$-linear maps $K'_3 \to A^2$  defined by 
\begin{align*}
\Sb'_3 \coloneqq 
\Big\{ 
F_3(b \ot a \ot b; b^2), 
F_3(b \ot c \ot b; b^2),  
F_3(b \ot e; b^2),  
F_3(e \ot b; a^2), 
F_3(b \ot e; a^2), \\
F_3(a \ot b \ot a; a^2), 
F_3(e \ot b; c^2), 
F_3(b \ot e; c^2), 
F_3(c \ot b \ot a; ac), 
F_3(b \ot e; ac)
\Big\}. 
\end{align*}
We observe that each element of $\Sb'_3$ occurs exactly once as a term in one of the maps in $\Sb_3$. 
Since 
$\Ah_2 \coloneqq \{a^2, b^2, c^2, ac\} \subseteq A^2$ 
is a basis of $A^2$, it follows that 
the set $\Sb_3$ is linearly independent.  
It remains to show that $\Sb_3$ generates $\Ker \partial^3$.  
Let $f \in \Ker \partial^3$ be an arbitrary element.   
For $v \in \Bh'_3$ and 
$x \in \Ah_2$, 
let $\varphi(v;x) \in \F_2$
be the coefficient 
such that, for all $v \in \Bh'_3$, 
\[
f(v) = \sum_{x \in \Ah_2} \varphi(v;x)x. 
\]
Again, since each element of $\Sb'_3$ occurs exactly once as a term in one of the maps in $\Sb_3$, 
we can assume 
by adding a linear combination of the elements of $\Sb_3$ to $f$  
that the coefficients 
\begin{align*}
\varphi(b \ot a \ot b; b^2), 
\varphi(b \ot c \ot b; b^2),  
\varphi(b \ot e; b^2),  
\varphi(e \ot b; a^2), 
\varphi(b \ot e; a^2), \\
\varphi(a \ot b \ot a; a^2), 
\varphi(e \ot b; c^2), 
\varphi(b \ot e; c^2), 
\varphi(c \ot b \ot a; ac), 
\varphi(b \ot e; ac)
\end{align*}
are all zero. 
Now it suffices to show that then $f$ must be the zero map. 
Since $f \in \Ker \partial^3$, 
$f$ must  satisfy
the relations in Lemma~\ref{kernel-rels}.
We see that \eqref{i} implies
that $\varphi(b \ot (a|c) \ot b;x) = 0$ 
for $x \neq b^2$. 
Hence we have 
$f(b \otimes (a|c) \otimes b) = 0$.  
We also see that \eqref{vii} implies that
$\varphi(b\otimes e; b^2) = \varphi(e \otimes b; b^2) = 0$.
We therefore have
$f(b\otimes e)=0$.  
Equation \eqref{v} implies that 
$\varphi(a\otimes b \otimes a; c^2) =0$, 
and \eqref{ii} implies that
$\varphi(a\otimes b \ot a; b^2) =0$. 
This shows that
either $f(a \otimes b \otimes a) = ac$ or
$f(a\otimes b \otimes a) = 0$.  
If
$f(a \otimes b \otimes a) = ac$, then
\eqref{v} implies that
$f(a\otimes b \otimes c) = c^2$. 
Now \eqref{iv} forces $f(e \otimes b) = c^2$, 
which contradicts $\varphi(e \ot b; c^2) = 0$. 
We thus have
$f(a \otimes b \otimes a) = 0$.  
From \eqref{v} and \eqref{ii}, we then get
$f(a \otimes b \otimes c) = 0$.  
From \eqref{iii} we deduce that
$\varphi(e \ot b;ac) = 
    \varphi(c \ot b \ot a; ac) = 0$,
and hence  
$f(e \otimes b) = 0$.  
It now follows from \eqref{iii} and \eqref{ii}
that
$f(c \otimes b \otimes a) = 0$.  
Finally, from \eqref{vi} and \eqref{ii} we get 
$f(c \otimes b \otimes c) = 0$.  
This shows that $f$ vanishes on all elements of the basis $\Bh'_3$, 
and hence $f$ is the zero map.  
\end{proof}
Since the elements in the set $\Sb_3$ belong to the subset $\Imm \partial^2$ of $\Ker \partial^3$, 
Lemma \ref{lemma:basis_of_Kerd3} implies $\Imm \partial^2 = \Ker \partial^3$. 
This finishes the proof of Theorem \ref{thm:computationof_HH}. 
\end{proof}


\end{document}